\documentclass[11pt,a4paper,reqno]{amsart}
\usepackage{clrscode}
\usepackage{mathrsfs,amssymb}

\newcommand{\scrA}{\mathscr{A}}
\newcommand{\scrR}{\mathscr{R}}
\newcommand{\ol}[1]{\overline{#1}}

\newtheorem{theorem}{Theorem}

\newtheorem{corollary}{Corollary}

\newtheorem{proposition}{Proposition}
\newtheorem{lemma}{Lemma}

\begin{document}
\begin{center}

   \title[Small Overlap Monoids]{Small Overlap Monoids: \\ The Word Problem}
   \maketitle

    Mark Kambites \\

    \medskip

    School of Mathematics, \ University of Manchester, \\
    Manchester M13 9PL, \ England.

    \medskip

    \texttt{Mark.Kambites@manchester.ac.uk} \\

    \medskip

\end{center}

\begin{abstract}
We develop a combinatorial approach to the study of semigroups and monoids 
with finite presentations satisfying small overlap conditions. In contrast 
to existing geometric methods, our approach facilitates a sequential
left-right analysis of words which lends itself to the development of
practical, efficient computational algorithms. In particular, we obtain a highly
practical linear time solution to the word problem for monoids and semigroups
with finite presentations satisfying the condition $C(4)$, and a polynomial
time solution to the uniform word problem for presentations satisfying the
same condition.
\end{abstract}

Small overlap conditions are simple and natural combinatorial conditions
on semigroup and monoid presentations, which serve to limit the complexity of derivation
sequences between equivalent words in the generators.
They form a natural semigroup-theoretic analogue of the \textit{small
cancellation} conditions which are extensively used in combinatorial
and computational group theory \cite{Lyndon77}. It is well known that every group admitting
a finite presentation satisfying suitable small cancellation conditions
is \textit{word hyperbolic} in the sense of Gromov \cite{Gromov87},
and in particular has word problem solvable in linear time.

In the 1970s, Remmers \cite{Remmers71, Remmers80} developed an
elegant geometric theory of small overlap semigroups, using the
natural semigroup-theoretic analogue of the \textit{van Kampen diagrams}
extensively employed in combinatorial group theory (see for example
\cite{Lyndon77}). He applied his methods to show that semigroups satisfying
sufficiently small overlap conditions have what would now be called
\textit{linear Dehn function}, that is, that the minimum length of a
derivation sequence between any two equivalent words is bounded above by
a linear function of the word lengths. In theory, it follows immediately
that one can
test if two words in the generators for such a semigroup are equivalent,
by exhaustively searching the
(finite) space of all applicable derivation sequences of the given length, to
see if any of them transforms one word to the other. However,
the number of possible derivation sequences, and hence the time complexity
of this algorithm, is exponential in the word length. More sophisticated
techniques (such as applications of graph reachability algorithms)
are of course applicable, but the problem remains one of searching a space of
exponential size, and so we cannot really hope that this approach will lead
to a tractable solution for the word problem.
The question naturally arises, then, of how hard the word problem really is
in these semigroups.

In this paper, we develop a new approach to the study of this important
class of semigroups and monoids, along purely combinatorial lines. While our work
lacks some of the mathematical elegance of Remmers' approach --- indeed
our foundational results are of a rather technical nature and our proofs
mainly by case analysis --- it has the
advantage of permitting a \textit{sequential} (left-right) analysis of
elements, which for computational purposes seems more relevant than a
geometric viewpoint. Two computational consequences of the theory we
develop are of particular interest. The first is a linear time (on a
two-tape Turing machine) algorithm to solve the word problem in any semigroup
with a presentation satisfying Remmers' condition $C(4)$. The second is a
polynomial time (more precisely, in the RAM model, quadratic in the presentation length and
linear in the word length) solution to the uniform
word problem for presentations satisfying the same condition.
While the proofs of correctness and of the time complexity bounds
for these algorithms are rather technical, the algorithms themselves are
quite straightforward to describe and eminently suitable for practical
implementation; the author is currently working on an implementation for
the GAP computer algebra system \cite{GAP}.

In addition to this introduction, this paper comprises five sections.
In Section~\ref{sec_prelim} we briefly recall the definitions of small
overlap semigroups and monoids, together with some of their properties,
and introduce some notation and terminology which will be used in the
rest of the paper. Section~\ref{sec_weakcan} establishes some technical,
but nonetheless important, combinatorial properties of small overlap
monoids, which are then used in Section~\ref{sec_equality} to give a
sequential characterisation of equivalence for two words in the generators
of a $C(4)$ presentation. Section~\ref{sec_algorithm} shows how this
characterisation can be used to develop a linear time algorithm for the
solution of the word problem of a fixed small overlap presentation. Finally,
in Section~\ref{sec_uniform} we apply our techniques to the solution of the
uniform word problem for $C(4)$ presentations; we also observe that
one test efficiently whether an arbitrary presentation satisfies the condition
$C(4)$.

The relationship of this work to the geometric approach developed
by Remmers \cite{Remmers71} perhaps deserves a further comment. As already
mentioned, our approach to small overlap semigroups is entirely combinatorial and, in its finished state, makes
no direct use of Remmers' geometric machinery. However, the author would most
likely never have arrived at this viewpoint without the insight and intuition
afforded by Remmers' approach, and the reader interested in fully understanding
the present paper may find it helpful to study also Remmers' work in parallel.
Some of his results have been given a very accessible treatment by Higgins
\cite{Higgins92}, but unfortunately the only complete source still seems to
be his thesis \cite{Remmers71}.

\section{Preliminaries}\label{sec_prelim}

We assume familiarity with basic notions of combinatorial semigroup
theory, including free semigroups and monoids, and semigroup and monoid
presentations.
In all but Section~\ref{sec_uniform} of the paper, which is devoted to uniform
decision problems, we assume we have a fixed finite presentatation for a
monoid
(or semigroup --- we
shall see shortly that the difference is unimportant). Words are
assumed to be drawn from the free monoid on the generating alphabet
unless otherwise stated. We write $u = v$ to indicate that two words are
equal in the free monoid, and $u \equiv v$ to indicate that they represent
the same element of the semigroup presented. We say that a word $p$ is a
\textit{possible prefix} of $u$ if there exists a (possibly empty) word
$w$ with $pw \equiv u$, that is, if the element represented by $u$ lies in
the right ideal generated by the element represented by $p$. The empty
word is denoted $\epsilon$.

A \textit{relation word} is a word which occurs as one side of a
relation in the presentation. A \textit{piece} is a word in the
generators which occurs as a factor in sides of two different relations,
or as a factor of both sides of a relation, or in two different (possibly
overlapping) places within one side of a relation. To ensure a uniform
treatment for free semigroups and monoids, we make the convention that the
empty word $\epsilon$ is always a piece, even if the presentation has no
relations.

The presentation is said to \textit{satisfy the condition $C(n)$}, where
$n$ is a positive integer, if no relation word can be written as the product
of \textbf{strictly fewer than} $n$ pieces. Thus for each $n$, $C(n+1)$ is
a strictly stronger condition than $C(n)$.
We briefly mention another related condition. The presentation \textit{satisfies
the condition OL($x$)}, where $0 \leq x \leq 1$ if whenever a piece $p$
occurs as a factor of a relation word $R$ we have $|p| < x |R|$.
Notice that if $n$ is a positive integer, then a semigroup
satisfying OL($1/n$) will certainly satisfy $C(n+1)$.

The weakest meaningful small overlap condition, $C(1)$, says
that no relation word is a product of zero pieces, that is, that $\epsilon$
is not a relation word. From this we see that in a small overlap monoid
presentation, no non-empty word can be equivalent to the empty word, that
is, no non-empty word can represent the identity. It follows that
every small overlap monoid presentation is also interpretable as a semigroup
presentation, and that the monoid presented is isomorphic to the semigroup
presented with an adjoined identity element. For simplicity in what follows
we shall focus upon small overlap monoids, but from each of our results one
can immediately deduce a corresponding result for small overlap semigroups.

For each relation word $R$, let $X_R$ and $Z_R$ denote respectively the
longest prefix of $R$ which is a piece, and the longest suffix of $R$
which is a piece. If the presentation satisfies $C(3)$ then $R$ cannot be
written as a product of two pieces, so this prefix and suffix cannot meet;
thus, $R$ admits a factorisation $X_R Y_R Z_R$ for some non-empty word $Y_R$.
If moreover the presentation satisfies the stronger condition $C(4)$ then $R$
cannot be written as a product of three pieces, so $Y_R$ is not a piece. The
converse also holds: a $C(3)$ presentation such that no $Y_R$ is a piece 
is a $C(4)$ presentation. We
call $X_R$, $Y_R$ and $Z_R$ the \textit{maximal piece prefix}, the
\textit{middle word} and the \textit{maximal piece suffix} respectively
of $R$.

Assuming now that the presentation satisfies at least the condition $C(3)$,
we shall use the letters $X$, $Y$ and $Z$ (sometimes with adornments or
subscripts) exclusively to represent maximal piece prefixes, middle words
and maximal piece suffixes respectively of relation words; two such letters
with the same subscript or adornment (or with none) will be assumed to
stand for the appropriate factors of the same relation word.

If $R$ is a relation word we write $\ol{R}$ for the (necessarily unique,
as a result of the small overlap condition)
word such that $(R, \ol{R})$ or $(\ol{R}, R)$ is a relation in the presentation. We write
$\ol{X_R}$, $\ol{Y_R}$ and $\ol{Z_R}$ for $X_{\ol{R}}$, $Y_{\ol{R}}$ and
$Z_{\ol{R}}$ respectively. (This is an abuse of notation since, for example,
the word $X_R$ may be a maximal piece prefix of two distinct relation words, but
we shall be careful to ensure that the meaning is clear from the context.)

\section{Weak Cancellation Properties}\label{sec_weakcan}

To perform efficient computations with words, it is very helpful to be
able to process them in a sequential, left-right manner. To facilitate
this in the case of the word problem for small overlap monoids, we
need to know what can be deduced about the equivalence (or non-equivalence)
of two words from prefixes of those words. This section
develops a theory with this end in mind, including a number of results which
can be viewed as weak cancellativity conditions satisfied by small overlap
monoids. We assume throughout a fixed monoid presentation satisfying the
small overlap condition $C(4)$.

We first introduce some terminology. A \textit{relation prefix} of a word
is a prefix which admits a (necessarily unique, as a consequence of the
small overlap condition) factorisation of the form $a X Y$ where $X$ and $Y$
are the maximal piece prefix and middle word respectively of some relation
word $XYZ$. An \textit{overlap prefix (of length $n$)} of
a word $u$ is a relation prefix which admits an (again necessarily unique)
factorisation of the form $b X_1 Y_1' X_2 Y_2' \dots X_n Y_n$ where
\begin{itemize}
\item $n \geq 1$;
\item no factor of the form $X_0Y_0$ begins before the end of the prefix $a$;
\item for each $1 \leq i \leq n$, $R_i = X_i Y_i Z_i$ is a relation word with $X_i$ and
$Z_i$ the maximal piece prefix and suffix respectively; and
\item for each $1 \leq i < n$, $Y_i'$ is a proper, non-empty prefix of $Y_i$.
\end{itemize}
Notice that if a word has a relation prefix, then the shortest such must
be an overlap prefix. A relation prefix $a X Y$ of a word $u$ is called
 \textit{clean} if $u$ does \textbf{not} have a prefix
$$a XY' X_1 Y_1$$
where $X_1$ and $Y_1$ are the maximal piece prefix and middle word respectively
of some relation word, and $Y'$ is a proper, non-empty prefix of $Y$. Clean
overlap prefixes, in particular, will play a crucial role in what follows.

\begin{proposition}\label{prop_overlapprefixnorel}
Let $a X_1 Y_1' X_2 Y_2' \dots X_n Y_n$
be an overlap prefix of some word. Then this prefix
contains no relation word as a factor (except possibly $X_n Y_n$ in
the case that $Z_n = \epsilon$).
\end{proposition}
\begin{proof}
Suppose that the given overlap prefix contains a
relation word $R$ as a factor. By the definition of an overlap prefix,
no occurrence of $R$ can begin before the end of the prefix $a$, so
we may assume that $R$ is a factor of $X_1 Y_1' X_2 Y_2' \dots X_n Y_n$.
It follows that either $R$ contains $X_i Y_i'$ as a factor for some $i$,
or else $R$ is a factor of $X_i Y_i' X_{i+1} Y_{i+1}'$ for some $i$ (where
$Y_{i+1}' = Y_n$ if $i+1 = n$) and we may assume without loss of generality
that the occurrence of $R$ overlaps non-trivially with the prefix $X_i Y_i'$.

In the former case, since $X_i$ is a maximal piece prefix of $X_i Y_i Z_i$
and $Y_i'$ is non-empty, $X_i Y_i'$ cannot be a piece; it follows then that
we must have $R = X_i Y_i Z_i$ with the occurrence in the obvious place. In
the latter case, $R$ is the product of a non-empty factor of $X_i Y_i Z_i$
with a factor of the $X_{i+1} Y_{i+1} Z_{i+1}$; but by the small overlap
assumption, $R$ cannot be written as a product of two pieces, so it must
again be that $R = X_i Y_i Z_i$ with the occurrence in the obvious place.

Now if $i = n$ then, since $R$ is a factor of the given relation prefix,
we must clearly have $R = X_i Y_i Z_i = X_i Y_i$ so that $Z_i = \epsilon$.
On the other hand, if $i < n$ then either $X_i Y_i Z_i$ contains
$X_{i+1} Y_{i+1}'$ as a factor, which
contradicts the fact that $X_{i+1}$ is a maximal piece prefix of $X_i Y_i Z_i$,
or else (recalling that $Y_i'$ is a proper prefix of $Y_i$) we see that
$X_{i+1} Y_{i+1}'$ contains a non-empty suffix of $Y_i$ followed by $Z_i$,
which contradicts the fact that $Z_i$ is a maximal piece suffix of
$X_i Y_i Z_i$.
\end{proof}

\begin{proposition}\label{prop_opgivesmop}
Let $u$ be a word. Every overlap prefix of $u$ is contained in a
clean overlap prefix of $u$.
\end{proposition}
\begin{proof}
We fix $u$ and prove by induction on the difference between the length of
$u$ and the length of the given overlap
prefix, that is, on the length of that part of $u$ not contained in the
given overlap prefix. For the base case, observe that an overlap prefix
constituting the whole of $u$ is necessarily clean. Now suppose
$a X_1 Y_1' \dots X_n Y_n$ is an overlap prefix, and that the result
holds for longer overlap prefixes of $u$. If the given prefix is clean then
there
is nothing to prove. Otherwise, by the definition of a clean overlap
prefix, there exist words $X$ and $Y$, being the
maximal piece prefix and the middle word respectively of some
relation word, and a proper non-empty prefix $Y_n'$ of $Y_n$ such that 
$$a X_1 Y_1' \dots X_n Y_n' XY$$
is a prefix of $u$. Clearly this is
an overlap prefix of $u$ which is strictly longer than the original one,
and so by induction is contained in a clean overlap prefix of $u$. But
now the original overlap prefix of is contained in a clean overlap prefix,
as required.
\end{proof}

\begin{corollary}\label{cor_nomopnorel}
If a word $u$ has no clean overlap prefix, then it contains
no relation word as a factor, and so if $u \equiv v$ then $u = v$.
\end{corollary}
\begin{proof}
Suppose $u$ has no clean overlap prefix. 
If $u$ contained a relation word as a factor
then clearly it would have a relation prefix, that is, a prefix of the
form $a X_R Y_R$ for some relation word $R$. But by our observations
above, the shortest relation prefix of $u$ would be an overlap prefix,
and so by Proposition~\ref{prop_opgivesmop}, is contained in a clean
overlap prefix of $u$. Thus, $u$ contains no relation word as a factor.
It follows easily that no relations can be applied to $u$, so the only
word equivalent to $u$ is $u$ itself.
\end{proof}

\begin{lemma}\label{lemma_staysclean}
If $u = w XYZ u'$ with $w XY$ a clean overlap prefix then
$w \ol{XY}$ is a clean overlap prefix of $w \ol{XYZ} u'$.
\end{lemma}
\begin{proof}
Let 
\begin{equation}\label{eqn_smop1}
w XY = a X_1 Y_1' \dots X_n Y_n' X Y
\end{equation}
be the factorisation given by the definition of a clean overlap prefix.
Then $w \ol{XYZ} u'$ has a prefix
\begin{equation}\label{eqn_smop2}
w \ol{XY} = a X_1 Y_1' \dots X_n Y_n' \ol{X} \ol{Y}
\end{equation}
If $n \geq 1$ it is immediate from the factorisation given
by \eqref{eqn_smop2} that $w \ol{XY}$ is an overlap prefix of $w \ol{XYZ} u'$.
In the case $n = 0$, however, we must consider the possibility that the
prefix $a \ol{XY} = w \ol{XY}$ contains a factor of the form $X_0 Y_0$
overlapping the final initial segment $a$. Suppose it does. Then recalling
that $Y_0$ is not a piece, and so cannot be a factor of $\ol{XY}$, we see
that $a \ol{XY}$ admits a factorisation
\begin{equation}
a \ol{XY} = b X_0 Y_0' \ol{XY}
\end{equation}
for some non-empty prefix $Y_0'$ or $Y_0$. Moreover, $Y_0'$ must be a proper
prefix of $Y_0$, or else $a$ would have a factor $X_0 Y_0$, contradicting
the fact that $w XY$ was a clean overlap prefix of $u$. This shows that
$w \ol{X} \ol{Y}$ is an overlap prefix of $w \ol{XYZ} u'$.

It remains to show that the given overlap prefix is clean. Suppose for a
contradiction that it is not. Then by definition, there is a factor of
the form $\hat{X} \hat{Y}$ overlapping
the end of the prefix $a \ol{XY}$; but this factor is either by contained in
$\ol{XYZ}$ (contradicting the supposition that $\hat{X}$ is a
maximal piece prefix of a relation word $\hat{X} \hat{Y} \hat{Z}$)
or contains a non-empty suffix of $\ol{Y}$ followed by $\ol{Z}$
(contradicting the assumption that $\ol{Z}$ is a maximal piece suffix of
$\ol{X Y Z}$). 
\end{proof}

The following lemma is fundamental to our approach to $C(4)$ monoids.
With careful application it seems to permit a comparable understanding to that resulting
from Remmers' geometric theory, but in a purely combinatorial (and
hence more computationally orientated) way.

\begin{lemma}\label{lemma_overlap}
Suppose a word $u$ has clean overlap prefix $w X Y$. If
$u \equiv v$ then $v$ has overlap prefix either
$w X Y$ or $w \ol{X Y}$, and no relation word occurring as a
factor of $v$ overlaps this prefix, unless it is $X Y Z$ or
$\ol{X Y Z}$ as appropriate.
\end{lemma}
\begin{proof}
Since $w X Y$ is an overlap prefix of $u$, it has by definition a
factorisation
$$w XY = a X_1 Y_1' \dots X_{n} Y_{n}' X Y$$
for some $n \geq 0$. We use this fact to prove the claim by induction on
the length $r$ of a rewrite sequence (using the defining relations) from
$u$ to $v$. 

In the case $r = 0$, we have $u = v$, so $v$ certainly has (clean) overlap
prefix $v XY$.
By Proposition~\ref{prop_overlapprefixnorel}, no relation word factor can occur entirely
within this prefix (unless it is $X Y$ and $Z = \epsilon$). If
a relation word factor of $v$ overlaps the end of the given overlap prefix
and entirely contains $XY$ then, since $XY$ is not a piece, that
relation word must clearly be $XYZ$. Finally,
a relation word cannot overlap the end of the given overlap prefix but
not contain the suffix $XY$, since this would clearly contradicts the
fact that the given overlap prefix is clean.

Suppose now for induction that the lemma holds for all values less than $r$,
and that there is a rewrite sequence from $u$ to $v$ of length $r$. Let
$u_1$ be the second term in the sequence, so that $u_1$ is obtained from
$u$ by a single rewrite using the defining relations, and $v$ from $u_1$
by $r-1$ rewrites.

Consider the relation word in $u$ which is to be rewritten in order to
obtain $u_1$, and in
particular its position in $u$. By Proposition~\ref{prop_overlapprefixnorel},
this relation word cannot be contained in the clean overlap prefix $w XY$,
unless it is $X Y$ where $Z = \epsilon$.

Suppose first that the relation word to be rewritten contains the final
factor $Y$
of the given clean overlap prefix. (Note that this covers in particular the
case that the relation word is $XY$ and $Z = \epsilon$.)
From the $C(4)$ assumption we know that $Y$ is not a piece, so we may deduce
that the relation word is $X Y Z$ contained in the obvious place. In
this case, applying the rewrite clearly leaves $u_1$ with a prefix 
$w \ol{XY}$, and by Lemma~\ref{lemma_staysclean}, this is a clean overlap
prefix.  Now $v$ can be obtained from
$u_1$ by $r-1$ rewrite steps, so it follows from the inductive hypothesis
that $v$ has overlap prefix either
$w \ol{XY}$ or $w \ol{\ol{XY}} = w XY$,
and that no relation word occurring as a factor of $v$ overlaps this
prefix, unless it is $X Y Z$ or $\ol{X Y Z}$ as appropriate; this
completes the proof in this case.

Next, we consider the case in which the relation word factor in $u$ to be
rewritten does not contain the final factor $Y_n$ of the clean overlap
prefix, but does overlap with the end of the clean overlap prefix. Then
$u$ has a factor of the form $X Y$, where $X$ is the maximal piece prefix
and $Y$ the middle word of a relation word, which overlaps $X_n Y_n$,
beginning after the start of $Y_n$. This clearly contradicts the assumption
that the overlap prefix is clean.

Finally, we consider the case in which the relation word factor in $u$
which is to be rewritten does not overlap the given clean overlap prefix
at all. Then obviously, the given clean overlap prefix of $u$ remains an
overlap prefix of $u_1$. If this overlap prefix is clean, then a simple
application of the inductive hypothesis again suffices to prove that $v$
has the required property.

There remains, then, only the case in which the given overlap prefix is
no longer clean in $u_1$. Then by definition there exist words $X$ and
$Y$, being a maximal piece prefix and middle word respectively of some relation
word, such
that $u_1$ has the prefix
$$a X_1 Y_1' \dots X_{n-1} Y_{n-1}' X_n Y_n' X Y$$
for some proper, non-empty prefix $Y_n'$ of $Y_n$.
Now certainly this is not a prefix of $u$, since this would contradict
the assumption that $a X_1 Y_1' \dots X_n Y_n$
 is a clean overlap
prefix of $u$. So we deduce that $u_1$ must contain a relation word
overlapping the final $XY$. This relation word cannot contain the final
factor $XY$, since this would again contradict the assumption that
$a X_1 Y_1' \dots X_n Y_n$ is a clean overlap prefix of $u$. Nor can
the relation word contain the final factor $Y$, since $Y$ is not a piece.
Hence, $u_1$ must have a prefix
$$a X_1 Y_1' \dots X_{n-1} Y_{n-1}' X_n Y_n' X Y' R$$
for some relation word and proper, non-empty prefix $Y'$ of $Y$ and
some relation word $R$. Suppose $R = X_R Y_R Z_R$ where $X_R$ and $Z_R$
are the maximal piece prefix and suffix respectively. Then it is readily
verified that
$$a X_1 Y_1' \dots X_{n-1} Y_{n-1}' X_n Y_n' X Y' X_R Y_R$$
is a clean overlap prefix of $u_1$. But now by the inductive
hypothesis, $v$ has prefix either
\begin{equation}\label{vprefix1}
a X_1 Y_1' \dots X_{n-1} Y_{n-1}' X_n Y_n' X Y' X_R Y_R
\end{equation}
or
\begin{equation}\label{vprefix2}
a X_1 Y_1' \dots X_{n-1} Y_{n-1}' X_n Y_n' X Y' \ol{X_R Y_R}
\end{equation}
and so in particular it certainly has prefix
$$a X_1 Y_1' \dots X_{n-1} Y_{n-1}' X_n Y_n' X Y'$$
which in turn is easily seen to have prefix
\begin{equation}\label{vprefix3}
a X_1 Y_1' \dots X_{n-1} Y_{n-1}' X_n Y_n.
\end{equation}
Moreover, by Proposition~\ref{prop_overlapprefixnorel}, the prefix
\eqref{vprefix1} or \eqref{vprefix2} of $v$ contains no relation
word as a factor (unless it is the final factor $X_R Y_R$ and
$Z_R = \epsilon$) and it follows easily that no relation word factor
overlaps the prefix \eqref{vprefix3} of $v$.
\end{proof}

The lemma has the following easy corollary.

\begin{corollary}\label{cor_noncleanprefix}
Suppose a word $u$ has (not necessarily clean) overlap prefix
$w XY$. If $u \equiv v$ then $v$ has a
prefix $w$ and contains no relation word overlapping this prefix.
\end{corollary}
\begin{proof}
By Proposition~\ref{prop_opgivesmop} the overlap prefix $wXY$
of $u$ is contained in a clean overlap prefix $w' X' Y'$ of $u$. Now
by Lemma~\ref{lemma_overlap}, $v$ has a prefix $w'$ and contains no relation
word overlapping this prefix. But it is easily seen that $w'$ must be at
least as long as $w$, so that $v$ has a prefix $w$ and contains no relation
word overlapping this prefix, as required.
\end{proof}

The following proposition describes a very weak left cancellation property
of small overlap monoids; it will allow us to restrict
attention to words with a prefix of the form $X Y$ where $X$ and $Y$ are
the maximal piece prefix and middle word respectively of some relation word.
\begin{proposition}\label{prop_dumpprefix}
Suppose a word $u$ has an overlap prefix $a X Y$ and that
$u = a X Y u''$. Then $u \equiv v$ if and only if $v = a v'$ where
$v' \equiv X Y u''$.
\end{proposition}
\begin{proof}
Clearly if $v = av'$ with $v' \equiv X_1 Y_1 u''$ then it is immediate
that $v = av' \equiv a X_1 Y_1 u'' = v$.

Conversely, suppose $u \equiv v$. Since $a X Y$ is an overlap prefix,
by Proposition~\ref{prop_overlapprefixnorel} it cannot contain a relation word
starting before the end of $a$. By Corollary~\ref{cor_noncleanprefix},
$v$ has prefix $a$, say $v = a v'$. Now consider a rewrite sequence, using
the defining relations, from $u$ to $v$.
Again using Corollary~\ref{cor_noncleanprefix}, every term in this
sequence will have prefix $a$, and contain no relation word overlapping
this prefix. It follows that the same sequence of rewrites can be applied
to take $X_1 Y_1 u''$ to $v'$, so that $v' \equiv X_1 Y_1 u''$ as required.
\end{proof}

We now introduce some more terminology. Let $u$ be a word
with shortest relation prefix $a X Y$, and let $p$ be a piece. We say that
$u$ is \textit{$p$-inactive} if $p u$ has shortest relation prefix $p a X Y$
and \textit{$p$-active} otherwise. The following proposition describes another
weak cancellation property of small overlap monoids.

\begin{proposition}\label{prop_inactive}
Let $u$ be a word and $p$ a piece.
If $u$ is $p$-inactive then $p u \equiv v$ if and only if $v = p w$
for some $w$ with $u \equiv w$.
\end{proposition}
\begin{proof}
Suppose $u$ has shortest relation prefix $aXY$, so that $pu$ has
shortest relation prefix $paXY$. Suppose $u = aXY u''$.
If $pu \equiv v$ then by Proposition~\ref{prop_dumpprefix} (since
the shortest relation prefix is clearly an overlap prefix), we
have $v = pa v'$ where $v' \equiv XY u''$. Now setting $w = av'$ we
have $v = pw$ and $u = a XY u' \equiv av' = aw$.
The converse implication is obvious.
\end{proof}

\begin{proposition}\label{prop_coactive}
Let $Z_1$ and $Z_2$ be maximal piece suffixes of relation words and suppose
$u$ is $Z_1$-active and $Z_2$-active. Then $Z_1$ and $Z_2$ have a
common non-empty suffix, and if $z$ is the maximal common suffix then
\begin{itemize}
\item[(i)] $u$ is $z$-active;
\item[(ii)] $Z_1 u \equiv v$ if and only if $v = z_1 v'$ where $z_1 z = Z_1$ and
            $v' \equiv z u$; and
\item[(iii)] $Z_2 u \equiv v$ if and only if $v = z_2 v'$ where $z_2 z = Z_2$;
            and $v' \equiv z u$.
\end{itemize}
\end{proposition}
\begin{proof}
Let $b X_3 Y_3$ and $c X_4 Y_4$ be the shortest relation prefixes of
$Z_1 u$ and $Z_2 v$ respectively. Since $u$ is $Z_1$-active and $Z_2$-active,
we must have $|b| < |Z_1|$ and $|c| < |Z_2|$. Moreover, since $Z_1$ is a
piece and $X_3$ is a maximal piece prefix of the relation word
$X_3 Y_3 Z_3$ we must have $|Z_1| \leq |b X_3|$, and similarly
$|Z_2| \leq |c X_4|$.

It follows that $u$ has prefixes $X_3' Y_3$ and $X_4' Y_4$ where
$X_3'$ and $X_4'$ are proper (perhaps empty) suffixes of $X_3$ and
$X_4$ respectively. Thus, one of $X_3' Y_3$ and $X_4' Y_4$ is a prefix
of the other, and so either $Y_3$ is a factor of
$X_4' Y_4$ and hence of $X_4 Y_4 Z_4$ or $Y_4$ is a factor of
$X_3' Y_3$ and hence of $X_3 Y_3 Z_3$. But by the $C(4)$ assumption, neither
$Y_3$ nor $Y_4$ is a piece so the only possible explanation is that
$X_3 Y_3 Z_3$ and $X_4 Y_4 Z_4$ are the same relation word, and
moreover $X_3' = X_4'$.

Now let $p$ be such that $p X_3' = X_3$. We have already observed that
$X_3'$ is a proper prefix of $X_3$, so $p$ is non-empty. Now
$Z_1 = bp$, and also $$p X_4' = p X_3' = X_3 = X_4$$
so by symmetry we have $Z_2 = cp$. Hence, $p$ is a common non-empty suffix of 
$Z_1$ and $Z_2$.

Now let $z$ be the maximal common suffix of $Z_1$ and $Z_2$. Let $y$, $z_1$ and
$z_2$ be such that $z = yp$, $Z_1 = z_1 z$ and $Z_2 = z_2 z$. Then clearly
$b = z_1 y$ and $c = z_2 y$. Now $zu = ypu$ has a relation prefix $y X_3 Y_3$,
from which it is immediate that $u$ is $z$-active so that (i) holds.

To show that (ii) holds, let $u'$ be such that $u = X_3' Y_3 u'$, and
suppose $u \equiv v$. Now
$$Z_1 u \ = \ z_1 z X_3' Y_3 u' \ = \ z_1 y p X_3' Y_3 u' \ = \ z_1 y X_3 Y_3 u'$$
where $z_1 y X_3 Y_3$ is the shortest relation prefix, and hence is an
overlap prefix. Hence, by Proposition~\ref{prop_dumpprefix} we have
$v = z_1 y v''$ where $v'' \equiv X_3 Y_3 u'$. But now setting
$v' = yv''$ we have $v = z_1 v'$, $z_1 z = Z_1$ and
$$v' \ = \ yv'' \ \equiv \ y X_3 Y_3 u' \ = \ y p X_3' Y_3 u' \ = \ z X_3' Y_3 u' \ = \ zu$$
as required. Conversely, if $v = z_1 v'$ where $z_1 z = Z_1$ and $v' \equiv zu$
then we have
$$Z_1 u \ = \ z_1 z u \ \equiv \ z_1 v' \ = \ v.$$
This completes the proof that (ii) holds, and an entirely symmetric argument shows that (iii) holds.
\end{proof}

\begin{corollary}\label{cor_actsame}
Let $Z_1$ and $Z_2$ be
maximal piece suffixes of relation words.
Suppose $u$ is $Z_2$-active and $Z_1 u \equiv Z_1 v$. Then
$Z_2 u \equiv Z_2 v$.
\end{corollary}
\begin{proof}
If $u$ is $Z_1$-inactive then by Proposition~\ref{prop_inactive} we have
$u \equiv v$, and so certainly $Z_2 u \equiv Z_2 v$.

On the other hand, if $u$ is $Z_1$-active then let $z$ be the maximal
common suffix of $Z_1$ and $Z_2$ and let $z_1$ and $z_2$ be such that
$z_1 z = Z_1$ and $z_2 z = Z_2$. Then by the Proposition~\ref{prop_coactive}(ii),
since
$Z_1 u \equiv Z_1 v$ we have $Z_1 v = z_1 v'$ where $v' \equiv zu$. But
from $z_1 z v = Z_1 v = z_1 v'$ we deduce that $v' = zv$, so now we have
$$Z_2 u \ = \ z_2 z u \ \equiv \ z_2 v' \  = \ z_2 z v \ = \  Z_2 v.$$
\end{proof}

\begin{corollary}\label{cor_eitheror}
Let $u$ and $v$ be words and $Z_1$ and $Z_2$ be maximal piece suffixes of
relation words. Suppose there exist words $u = u_1, \dots, u_n = v$ such that
\begin{align*}
Z_1 u_1 \equiv Z_1 u_2, \ Z_2 u_2 \equiv Z_2 u_3, \ &Z_1 u_3 \equiv Z_1 u_4, \ \dots \\
&\dots , \ 
\begin{cases}
Z_1 u_{n-1} \equiv Z_1 u_n &\text{ if $n$ is even} \\
Z_2 u_{n-1} \equiv Z_2 u_n &\text{ if $n$ is odd}.
\end{cases}
\end{align*}
Then either $Z_1 u \equiv Z_1 v$ or $Z_1 u \equiv Z_2 v$ or both.
\end{corollary}
\begin{proof}
Fix $u$ and $v$, and suppose $n$ is minimal (allowing exchanging $Z_1$
and $Z_2$ if necessary) such that a sequence of equivalences as above
exists. Suppose further for a contradiction that $n > 2$. If $u_2$ was
$Z_1$-inactive then by
Proposition~\ref{prop_inactive} we would have $u_1 \equiv u_2$
so that $Z_2 u_1 \equiv Z_2 u_2 \equiv Z_2 u_3$, contradicting the
minimality assumption on $n$. Similarly, if $u_2$ was $Z_2$-inactive
then we would have $u_2 \equiv u_3$ so that $Z_1 u_1 \equiv Z_1 u_2 \equiv Z_1 u_3$
again contradicting the minimality assumption on $n$.

Thus, $u_2$ is both $Z_1$-active and $Z_2$-active.
But now since $Z_1 u_1 \equiv Z_1 u_2$, we apply
Corollary~\ref{cor_actsame} to see that
$Z_2 u_1 \equiv Z_2 u_2 \equiv Z_2 u_3$, again providing the required contradiction.
\end{proof}

\section{Sequential Characterisation of Equality}\label{sec_equality}

In this section we use the theory developed in Section~\ref{sec_weakcan}
to provide a new characterisation of when two words
in the generators of a small overlap presentation represent the same element
of the monoid presented. In Section~\ref{sec_algorithm} we shall use this 
characterisation to develop an efficient algorithm to solve the word problem.

We first present a lemma which gives a set of mutually exclusive combinatorial
conditions, the disjunction of which is necessary and sufficient for two words
of a certain form to represent the same element.

\begin{lemma}\label{lemma_eq}
Suppose $u = X Y u'$ where $XY$ is a clean overlap prefix of
$u$. Then $u \equiv v$ if and only if one of the following mutually
exclusive conditions holds:
\begin{itemize}
\item[(1)] $u = XYZ u''$ and $v = XYZ v''$ and either
$Z u'' \equiv Z v''$ or $\ol{Z} u'' \equiv \ol{Z} v''$ or both;
\item[(2)] $u = X Y u'$, $v = X Y v'$, and $Z$ fails to be a
prefix of at least one of $u'$ and $v'$, and $u' \equiv v'$;
\item[(3)] $u = X Y Z u''$, $v = \ol{X} \ol{Y} \ol{Z} v''$
and either $Z u'' \equiv Z v''$ or $\ol{Z} u'' \equiv \ol{Z} v''$
or both;
\item[(4)] $u = X Y u'$, $v = \ol{X} \ol{Y} \ol{Z} v''$ but
$Z$ is not a prefix of $u'$ and $u' \equiv Z v''$;
\item[(5)] $u = X Y Z u''$, $v = \ol{X} \ol{Y} v'$
but $\ol{Z}$ is not a prefix of $v'$ and $\ol{Z} u'' \equiv v'$;
\item[(6)] $u = X Y u'$, $v = \ol{X} \ol{Y} v'$, $Z$ is not
a prefix of $u'$ and $\ol{Z}$ is not a prefix of $v'$, but
$Z = z_1 z$, $\ol{Z} = z_2 z$, $u' = z_1 u''$, $v' = z_2 v''$ where
$u'' \equiv v''$ and $z$ is the maximal common suffix of $Z$ and $\ol{Z}$,
$z$ is non-empty, and $z$ is a possible prefix of $u''$.
\end{itemize}
\end{lemma}
\begin{proof}
First we treat the claim that the conditions (1)-(6) are mutually exclusive.
Since $X$ is a maximal piece prefix of $XYZ$ and $Y$ is non-empty, $XY$ is
not a piece. An entirely similar argument shows that $\ol{XY}$ is not a piece.
In particular, neither of $XY$ and $\ol{XY}$ is a prefix of the other, and
so $v$ can have at most one of them as a prefix. Thus, conditions (1)-(2)
are not consistent with conditions (3)-(6). The mutual exclusivity of (1)
and (2) is self-evident from the definitions, and likewise that of (3)-(6). 

It is easily verified that each of the conditions
(1)-(5) imply that $u \equiv v$. We show next that (6) implies that
$u \equiv v$. Since $z$ is a possible prefix of $u''$ and $u'' \equiv v''$,
we may write $u'' \equiv zx \equiv v''$ for some word $x$. Now we have
\begin{align*}
u = X Y u' = XY z_1 u'' &\equiv XY z_1 z x = XYZ x \\
&\equiv \ol{XYZ} x = \ol{XY} z_1 z x \equiv \ol{XY} z_2 v'' = \ol{XY} v' = v.
\end{align*}
What remains, which is the main burden of
the proof, is to prove that $u \equiv v$ implies that at least one of
the conditions (1)-(6) holds. To this end, then, suppose $u \equiv v$;
then there is a rewriting sequence taking $u$ to $v$. 
By Lemma~\ref{lemma_overlap}, every term in this sequence will have prefix
either $XY$ or $\ol{XY}$ and this prefix
can only be modified by the application of the relation $(XYZ, \ol{XYZ})$
in the obvious place. We now prove the claim by case analysis.

By Lemma~\ref{lemma_overlap}, $v$ begins either with $XY$ or with $\ol{XY}$.
Consider first the case in which $v$ begins with $XY$; we split this into
two further cases depending on whether $u$ and $v$ both begin with the full
relation word $XYZ$; these will correspond respectively to conditions (1)
and (2) in the statement of the lemma.

\textbf{Case (1).} Suppose $u = XYZ u''$ and $v = X Y Z v''$.
Then clearly there is a rewriting sequence taking $u$ to $v$ which by
Lemma~\ref{lemma_overlap} can be
broken up as:
\begin{align*}
u = XYZ u'' \to^* X Y Z u_1 \to &\ol{XYZ} u_1 \to^* \ol{XYZ} u_2 \\
&\to XYZ u_2 \to^* \dots \to XYZ u_n \to^* XYZ v'' = v
\end{align*}
where none of the steps in the sequences indicated by $\to^*$ involves rewriting
a relation word overlapping with the prefix $XY$ or
$\ol{XY}$ as appropriate. It follows that there are rewriting sequences.
$$Z u'' \to^* Z u_1, \ \ol{Z} u_1 \to^* \ol{Z} u_2, \ Z u_2 \to^* Z u_3, \ \dots, \ Z u_n \to^* Z v''$$
Now by Corollary~\ref{cor_eitheror}, either $Z u'' \equiv Z v''$ or $\ol{Z} u'' \equiv \ol{Z} v''$
as required to show that condition (1) holds.

\textbf{Case (2).} Suppose now that $u = X Y u'$, $v = XY v'$ and $Z$
fails to be a prefix of at least one of $u'$ and $v'$. We must show that
$u' \equiv v'$; suppose for a contradiction that this does not hold. We
consider only the case that $Z$ is not a prefix of $u'$; the case that
$Z$ is not a prefix of $v'$ is symmetric. We consider rewriting sequences
from $u = XY u'$ to $v = XY v'$. Again using Lemma~\ref{lemma_overlap}, we
see that there is either (i) such a sequence taking $u$ to $v$ containing
no rewrites of relation words overlapping the prefix $XY$, or (ii) such a
sequence taking $u$ to $v$ which can be broken up as:
\begin{align*}
u = X Y u' \to^* XYZ u_1 \to &\ol{XYZ} u_1 \to^* \ol{XYZ} u_2 \\
&\to XYZ u_2 \to^* \dots \to XYZ u_n \to^* XY v' = v
\end{align*}
where none of the intermediate words in the sequences indicated by $\to^*$
contains a relation word overlapping with the prefix $XY$ or
$\ol{XY}$ as appropriate. In case (i) there is clearly a rewrite sequence
taking $u'$ to $v'$ so that $u' \equiv v'$ as required. In case (ii), there
are rewriting sequences.
$$u' \to^* Z u_1, \ \ol{Z} u_1 \to^* \ol{Z} u_2, \ Z u_2 \to^* Z u_3, \ \dots, \ Z u_n \to^* v'.$$
Notice that, since $u'$ does not begin with $Z$, we can deduce from
Proposition~\ref{prop_inactive} that $u_1$ is $Z$-active.
By Corollary~\ref{cor_eitheror}, either $Z u_1 \equiv Z u_n$ or
$\ol{Z} u_1 \equiv \ol{Z} u_n$. In the latter case, since $u_1$ is
$Z$-active, Corollary~\ref{cor_actsame} tells us that we also have
$Z u_1 \equiv Z u_n$ in any case. But now
$$u' \equiv Z u_1 \equiv Z u_n \equiv v'$$
so condition (2) holds and we are done.

We have now shown that if $v$ begins with $XY$ then either condition (1)
or condition (2) holds. It remains to consider the case in which $v$ begins
with $\ol{XY}$, and show that one of conditions (1)-(6) must be satisfied.
We split the analysis here into four cases depending on whether $u$ begins 
with the full relation word $XYZ$, and whether $v$ begins with the full
relation word $\ol{XYZ}$; these four cases will correspond respectively to 
conditions (3)-(6) in the statement of the lemma.

\textbf{Case (3).} Suppose $u = XYZ u''$ and
$v = \ol{XYZ} v''$.
Then $u = XYZ u'' \equiv v \equiv XYZ v''$, so by the same argument as in case (1) we
have either $Zu'' \equiv Z v''$ or $\ol{Z} u'' \equiv \ol{Z} v''$ as required
to show that condition (3) holds.

\textbf{Case (4).} Suppose $u = XY u'$ and
 $v = \ol{XYZ} v''$ but $Z$ is not a prefix of $u'$. Then
$u = XY u' \equiv v \equiv XYZ v''$. Now applying the same argument as
in case (2) (with $XYZ v''$ in place of $v$ and setting $v' = Zv''$) we
have $u' \equiv v' = Z v''$ so that condition (4) holds.

\textbf{Case (5).} Suppose $u = XYZ u''$, $v = \ol{XY} v'$
but $\ol{Z}$ is not a prefix of $v'$. Then we have
$\ol{XYZ} u'' \equiv u \equiv v = \ol{XY} v'$. Now applying the same
argument as in case (1) (but with $\ol{XYZ} u''$ in place of $u$ and
setting $u' = \ol{Z} u''$) we
obtain $u' \equiv v' = \ol{Z} u''$ so that condition (5) holds.

\textbf{Case (6).} Suppose $u = XY u'$, $v = \ol{XY} v'$ and that $Z$ is not a
prefix of $u'$ and $\ol{Z}$ is not a prefix of $v'$.
It follows this time there is a rewriting sequence taking $u$ to $v$ of
the form
\begin{align*}
u = X Y u' \to^* XYZ u_1 &\to \ol{XYZ} u_1 \to^* \ol{XYZ} u_2 \to XYZ u_2 \\
&\to^* \dots \to \ol{XYZ} u_n \to^* \ol{XY} v' = v
\end{align*}
where once more none of the intermediate words in the sequences indicated by $\to^*$
contains a relation word overlapping with the prefix $XY$ or
$\ol{XY}$ as appropriate.
Now there are rewriting sequences.
$$u' \to^* Z u_1, \ol{Z} u_1 \to^* \ol{Z} u_2, Z u_2 \to^* Z u_3, \dots, Z u_{n-1} \to^* Z u_n, \ol{Z} u_n \to^* v'.$$
Notice that, since $u'$ does not begin with $Z$, we may deduce from Proposition~\ref{prop_inactive}
that $u_1$ is $Z$-active.
By Corollary~\ref{cor_eitheror}, either $Z u_1 \equiv Z u_n$ or
$\ol{Z} u_1 \equiv \ol{Z} u_n$. In the latter case, since $u_1$ is
$Z$-active, Corollary~\ref{cor_actsame} tells us that we also have
$Z u_1 \equiv Z u_n$ anyway. But now
$$u' \equiv Z u_1 \equiv Z u_n$$ where $u'$ does not begin with $Z$, and
also $v' \equiv \ol{Z} u_n$ were $v'$ does not begin with $\ol{Z}$. By
applying Proposition~\ref{prop_inactive} twice, we deduce that $u_n$ is both
$Z$-active and $\ol{Z}$-active.

Let $z$ be the maximal common suffix of $Z$ and $\ol{Z}$. Then
applying Proposition~\ref{prop_coactive} (with $Z_1 = Z$ and $Z_2 = \ol{Z}$),
we see that $z$ is non-empty and
\begin{itemize}
\item $u' = z_1 u''$ where $Z = z_1 z$ and $u'' \equiv z u_n$; and
\item $v' = z_2 v''$ where $\ol{Z} = z_2 z$ and $v'' \equiv z u_n$.
\end{itemize}
But then we have
$u'' \equiv z u_n \equiv v''$ and also $z$ is a possible prefix of
$u''$ as required to show that condition (6) holds.
\end{proof}

Lemma~\ref{lemma_eq} gives a first clue as to how one might solve the word
problem for a small overlap monoid by analysing words sequentially from left
to right. The natural strategy is as follows. First, use Proposition~\ref{prop_dumpprefix} to reduce to the case
in which the words both have clean relation prefixes of the form $XY$ or $\ol{XY}$.
Now by examining short prefixes, one can clearly always rule out at least
five of the six mutually exclusive conditions of the lemma. The remaining
condition will involve equivalence of words derived from suffixes of $u$
and $v$, so apply the same approach recursively to test whether this
condition is satisfied.

This approach meets with several apparent obstacles. Firstly, it
is not clear that the words derived from the suffixes of $u$ and $v$, which
must be tested for equivalence in the recursive call, are shorter than the
original words $u$ and $v$; for example, a relation word $XYZ$ may be shorter
than the maximal piece suffix $\ol{Z}$ of the word on the other side of the
relation. In fact the recursive call will not always involve shorter words,
but it will involve words which are simpler in a more subtle sense, so
that the algorithm still terminates rapidly. Secondly, some of the conditions
involve a disjunction of equivalence of \textit{two} pairs of words derived
from the suffixes; testing both would require two recursive calls, potentially
leading to exponential time complexity. It tranpires, though, that the
theory of activity and inactivity developed in Section~\ref{sec_weakcan}
means that one recursive call will always suffice.
Finally, some of the conditions require us to check the possible prefixes
of words derived from suffixes; this problem is solved by the following
development of Lemma~\ref{lemma_eq}, which gives simultaneous conditions for
two words to be equal, and to admit a given piece as a possible prefix.

\begin{lemma}\label{lemma_eqandprefix}
Suppose $u = X Y u'$ where $XY$ is a clean overlap prefix, and suppose
$p$ is a piece. Then $u \equiv v$ and $p$ is a possible prefix of $u$
if and only if one of the following mutually exclusive conditions holds:
\begin{itemize}
\item[(1')] $u = XYZ u''$ and $v = XYZ v''$, either
$Z u'' \equiv Z v''$ or $\ol{Z} u'' \equiv \ol{Z} v''$, and
also $p$ is a prefix of either $X$ or $\ol{X}$ or both;

\item[(2')] $u = X Y u'$, $v = X Y v'$, and $Z$ fails to be a
prefix of at least one of $u'$ and $v'$, and $u' \equiv v'$,
and also either
\begin{itemize}
\item $p$ is a prefix of $X$
\item $p$ is a prefix of $\ol{X}$ and $Z$ is a possible prefix of $u'$;
\end{itemize}
or both;

\item[(3')] $u = X Y Z u''$, $v = \ol{X} \ol{Y} \ol{Z} v''$
and either $Z u'' \equiv Z v''$ or $\ol{Z} u'' \equiv \ol{Z} v''$
or both, and also $p$ is a prefix of $X$ or $\ol{X}$ or both;

\item[(4')] $u = X Y u'$, $v = \ol{X} \ol{Y} \ol{Z} v''$ but
$Z$ is not a prefix of $u'$ and $u' \equiv Z v''$, and also
$p$ is a prefix of $X$ or $\ol{X}$ or both;

\item[(5')] $u = X Y Z u''$, $v = \ol{X} \ol{Y} v'$
but $\ol{Z}$ is not a prefix of $v'$ and $\ol{Z} u'' \equiv v'$,
and also $p$ is a prefix of $X$ or $\ol{X}$ or both;

\item[(6')] $u = X Y u'$, $v = \ol{X} \ol{Y} v'$, $Z$ is not
a prefix of $u'$ and $\ol{Z}$ is not a prefix of $v'$, but
$Z = z_1 z$, $\ol{Z} = z_2 z$, $u' = z_1 u''$, $v' = z_2 v''$ where
$u'' \equiv v''$, $z$ is the maximal common suffix of $Z$ and $\ol{Z}$,
$z$ in non-empty, $z$ is a possible prefix of $u''$, and
also $p$ is a prefix of $X$ or $\ol{X}$ or both.
\end{itemize}
\end{lemma}
\begin{proof}
Mutual exclusivity of the six conditions is proved exactly as for
Lemma~\ref{lemma_eq}.

Suppose now that one of the six conditions above applies. Each condition
clearly implies the corresponding condition from Lemma~\ref{lemma_eq},
so we deduce immediately that $u \equiv v$. We must show, using the fact
that $p$ is a prefix of $X$ or of $\ol{X}$, that $p$ is a possible prefix
of $u$, or equivalently of $v$.

In case (1'), if $p$ is a prefix of $X$ then it is a prefix of $u$, while
if $p$ is a prefix of $\ol{X}$ then it is a prefix of $\ol{XYZ} u''$ which
is clearly equivalent to $u$.  In case (2'), if $p$ is a prefix of $X$ then
it is again a prefix of $u$, while if $p$ is a prefix of $\ol{X}$ and $Z$ is a
possible prefix of $u'$, say $u' \equiv Z w$, then
$$u \ = \ XYu' \ \equiv \ XYZw \ \equiv \ \ol{XYZ} w$$
where the latter has $p$ as a prefix. In the remaining cases $u$ begins
with $X$ and $v$ begins with $\ol{X}$, so $p$ is a prefix of either $u$
or $v$, and hence a possible prefix of $u$.

Conversely, suppose $u \equiv v$ and $p$ is a possible prefix of $u$. Then
exactly one of the six conditions in Lemma~\ref{lemma_eq} applies. By
Lemma~\ref{lemma_overlap}, every word equivalent to $u$ begins with either
$XY$ or $\ol{XY}$. Since $p$ is a piece, $X$ is the maximal piece prefix
of $XYZ$, and $\ol{X}$ is the maximal piece prefix of $\ol{XYZ}$ it follows
that $p$ is a prefix of either $X$ or $\ol{X}$. If any but condition (2)
of Lemma~\ref{lemma_overlap} is satisfied, this suffices to show
that the corresponding condition from the statement of
Lemma~\ref{lemma_eqandprefix} holds.

If condition (2) from Lemma~\ref{lemma_eq} applies, we must show
additionally that either $p$ is a prefix of $X$, or $p$ is a prefix
of $\ol{X}$ and $Z$ is a possible prefix of $u'$. Suppose $p$ is not
a prefix of $X$. Then by the above, $p$ is a prefix of $\ol{X}$. It follows from Lemma~\ref{lemma_overlap}, that the
only way the prefix $XY$ of the word $u$ can be changed using the defining
relations is by application of
the relation $(XYZ, \ol{XYZ})$. In order for this to happen, one must
clearly be able to rewrite $u = XYu'$ to a word of the form $XYZ w$;
consider the shortest possible rewriting sequence which achieves this.
By Lemma~\ref{lemma_overlap}, no term in the sequence except for the last
term will contain a relation word overlapping the initial $XY$. It follows
that the same rewriting steps rewrite $u'$ to $Zw$, so that $Z$ is a
possible prefix of $u'$, as required.
\end{proof}

\section{The Algorithm}\label{sec_algorithm}

In this section we present an algorithm, for a fixed monoid presentation
satisfying $C(4)$, which takes as input arbitrary words
$u$ and $v$ and a piece $p$, and decides whether $u \equiv v$ and $p$ is a
possible prefix of $u$. It will transpire that this algorithm can be
implemented to run time in linear in the shorter of $u$ and $v$. In particular,
by setting $p = \epsilon$ we obtain an algorithm to solve the word problem in
time linear in the smaller of the input words.
The algorithm is shown (in recursive/functional pseudocode) in
Figure~1. Our first objective is to prove
the correctness of the algorithm, that is, that whenever the algorithm
terminates, it provides the output it gives is correct.
\begin{figure}
\begin{codebox}
\Procname{$\proc{WP-Prefix}(u, v, p)$}
\li     \If $u = \epsilon$ or $v = \epsilon$ \label{li_start_a}
\li         \Then \If $u = \epsilon$ and $v = \epsilon$ and $p = \epsilon$
\li             \Then \Return \const{Yes}                \label{li_allepsilon}
\li             \Else \Return \const{No}                 \label{li_someepsilon}
            \End \label{li_end_a}
\li     \ElseIf $u$ does not have the form $XYu'$ with $XY$ a clean overlap prefix
\li     \Then \If $u$ and $v$ begin with different letters \label{li_start_b}
\li         \Then \Return \const{No}                     \label{li_uvdifferentstart}
\li        \ElseIf $p \neq \epsilon$ and $u$ and $p$ begin with
different letters
\li         \Then \Return \const{No}                     \label{li_updifferentstart}
\li         \ElseNoIf
\li       $u \gets u$ with first letter deleted
\li      $v \gets v$ with first letter deleted
\li      \If $p \neq \epsilon$
\li          \Then $p \gets p$ with first letter deleted
         \End
\li      \Return $\proc{WP-Prefix}(u,v,p)$   \label{li_rec_nomop}
\End \label{li_end_b}

\li \ElseNoIf
\li $\kw{let}\  X, Y, u'$ be such that $u = XY u'$ \label{li_start_c}

\li \If $p$ is a prefix of neither $X$ nor $\ol{X}$
\li \Then \Return \const{No} \label{li_pnotprefix}

\li \ElseIf $v$ does not begin either with $XY$ or with $\ol{XY}$
\li \Then \Return \const{No} \label{li_vstartswrong}

\li \ElseIf $u = XYZ u''$ and $v = XYZ v''$
\li    \Then \If $u''$ is $\ol{Z}$-active
\li       \Then \Return $\proc{WP-Prefix}(\ol{Z} u'', \ol{Z} v'', \epsilon)$ \label{li_rec_case1a}
\li       \Else \Return $\proc{WP-Prefix}(Z u'', Z v'', \epsilon)$ \label{li_rec_case1b}
       \End

\li \ElseIf $u = XY u'$ and $v = XY v'$
\li     \Then \If $p$ is a prefix of $X$
\li         \Then \Return $\proc{WP-Prefix}(u',v', \epsilon)$ \label{li_rec_case2a}
\li         \Else \Return $\proc{WP-Prefix}(u',v', Z)$ \label{li_rec_case2b}
        \End

\li \ElseIf $u = XYZ u''$ and $v = \ol{XYZ} v''$
\li     \Then \If $u''$ is $\ol{Z}$-active
\li         \Then \Return $\proc{WP-Prefix}(\ol{Z} u', \ol{Z} v', \epsilon)$ \label{li_rec_case3a}
\li         \Else \Return $\proc{WP-Prefix}(Z u', Z v', \epsilon)$ \label{li_rec_case3b}
        \End

\li \ElseIf $u = XY u'$ and $v = \ol{XYZ} v''$
\li     \Then \Return $\proc{WP-Prefix}(u', Z v'', \epsilon)$ \label{li_rec_case4}

\li \ElseIf $u = XYZ u''$ and $v = \ol{XY} v'$
\li     \Then \Return $\proc{WP-Prefix}(\ol{Z} u'', v', \epsilon)$ \label{li_rec_case5}

\li \ElseIf $u = XY u'$ and $v = \ol{XY} v'$
\li     \Then \kw{let} $z$ be the maximal common suffix of $Z$ and $\ol{Z}$
\li           \kw{let} $z_1$ be such that $Z = z_1 z$
\li           \kw{let} $z_2$ be such that $\ol{Z} = z_2 z$
\li           \If $u'$ does not begin with $z_1$ or $v'$ does not begin with $z_2$;
\li               \Then \Return \const{NO} \label{li_case6no}
\li               \Else \kw{let} $u''$ be such that $u' := z_1 u''$
\li                     \kw{let} $v''$ be such that $v' := z_2 v''$;
\li                     \Return $\proc{WP-Prefix}(u'', v'', z)$ \label{li_rec_case6} \label{li_end_c}
              \End
        \End
    \End
\end{codebox}
\caption{Algorithm for the Word Problem} \label{fig_algorithm}

\end{figure}

\begin{lemma}\label{lemma_correctness}
Suppose $u$ and $v$ are words and $p$ a piece. Then the algorithm
$\proc{WP-PREFIX}(u,v,p)$
\begin{itemize}
\item outputs $\const{YES}$ only if $u \equiv v$ and $p$ is a possible prefix of
$u$; and
\item outputs $\const{NO}$ only if $u \neq v$ or $p$ is not a possible prefix of $u$.
\end{itemize}
\end{lemma}
\begin{proof}
We prove correctness using induction on the number $n$ of recursive calls.

Consider first the base case $n = 0$, that is, where the algorithm terminates
without a recursive call. Suppose $u$, $v$ and $p$ are such that
this happens. We consider each of the possible lines at which termination
may occur, establishing in each case that the output produced is correct.

\begin{itemize}
\item[\textbf{Line \ref{li_allepsilon}.}] If $u = \epsilon$, $v = \epsilon$ and $p = \epsilon$
then clearly $u \equiv v$ and $p$ is a possible prefix of $u$, so the output
$\const{YES}$ is correct.

\item[\textbf{Line \ref{li_someepsilon}.}] If $u = \epsilon$ [respectively, $v = \epsilon$]
then it follows
easily from the small overlap condition $C(4)$ that no relations can be applied to
$u$ [$v$]; indeed a relation which could be applied to $u$ [$v$] would have to
have $\epsilon$ as one side, but $\epsilon$ is a piece and hence cannot be a
relation word. Hence, we can have that $u \equiv v$ and $p$ is a possible
prefix of $u$ only if $u = v = p = \epsilon$. In this case, this condition
is not satisfied, so the output $\const{NO}$ is correct.

\item[\textbf{Line \ref{li_uvdifferentstart}.}] In this case, $u$ does not begin with a clean overlap prefix of
the form $XY$. So by Proposition~\ref{prop_dumpprefix}, every word equivalent to $u$
must begin with the same letter as $u$. Hence, if $u$ and $v$ do not begin
with the same letter then we cannot have $u \equiv v$,
so the output $\const{NO}$ is correct.

\item[\textbf{Line \ref{li_updifferentstart}.}] Again, $u$ does not begin with a clean overlap prefix. If $p$ is
non-empty and begins with a different letter to $u$, then again by
Proposition~\ref{prop_dumpprefix}, $p$ cannot be a possible prefix of $u$,
so the output $\const{NO}$ is correct.

\item[\textbf{Line \ref{li_pnotprefix}.}] We are now in the case that
$u$ has a clean overlap prefix $XY$. If $p$ is not a prefix of $X$ or
$\ol{X}$ then by Lemma~\ref{lemma_eqandprefix} we see that $p$ is not a
possible prefix of $u$, so the output $\const{NO}$ is correct.

\item[\textbf{Line \ref{li_vstartswrong}.}] Once again, we are in the
case that $u$ has a clean overlap prefix $XY$. If $v$ does not begin
with either XY or $\ol{XY}$ then by Lemma~\ref{lemma_eq} we cannot have
$u \equiv v$ so the output $\const{NO}$ is correct.

\item[\textbf{Line \ref{li_case6no}.}] We are now in the case that
$u = XYu'$ and $v = \ol{XY} v'$ where $Z$ is not a prefix of $u'$
and $\ol{Z}$ is not a prefix of $v'$. We know also that $z$ is the
maximal common suffix of $Z$ and $\ol{Z}$ and $z_1$ and $z_2$ are such
that $Z = z_1 z$ and $\ol{Z} = z_2 z$. By Lemma~\ref{lemma_eqandprefix}
we cannot have $u \equiv v$ unless $u'$ and $v'$ have the form $z_1 u''$
and $z_2 v''$ respectively, so if this is not the case, the output
$\const{NO}$ is correct.
\end{itemize}

Now let $n > 0$ and suppose for induction that the algorithm produces the
correct output whenever it terminates after strictly fewer than $n$ recursive
calls. Let $u, v, p$ be such that the algorithm terminates after $n$
recursive calls. This time, we consider each of the possible places at which
the first recursive call can be made, establishing in each case that the output
produced is correct.
\begin{itemize}
\item[\textbf{Line \ref{li_rec_nomop}.}] In this case $u$ does not begin with a clean overlap prefix of the
form $XY$ and we have $u = au'$. It follows by Proposition~\ref{prop_dumpprefix} that
every word equivalent to $u$ has the form $aw$ where $w \equiv u'$. In
particular, $u \equiv v = av'$ if and only if $u' \equiv v'$, $p$ is a
possible prefix exactly if either $p = \epsilon$ or $p = a p'$ where
$p'$ is a possible prefix of $u'$. By the inductive hypothesis, the
recursive call correctly establishes whether these conditions hold.

\item[\textbf{Line \ref{li_rec_case1a}.}] 
We know that $u = XYZ u''$, that $v = XYZ v''$ and that $p$ is a prefix of $X$ or $\ol{X}$. By Lemma~\ref{lemma_eqandprefix},
it follows that $u \equiv v$ and $p$ is a possible prefix of $u$ if and only
if $Z u'' \equiv Z v''$ or $\ol{Z} u'' \equiv \ol{Z} v''$. We also know
that $u''$ is $\ol{Z}$-active, so by Corollary~\ref{cor_actsame}, this is
true if and only if $\ol{Z} u'' \equiv \ol{Z} v''$.

\item[\textbf{Line \ref{li_rec_case1b}.}] This is the same as the previous
case, except that $u''$ is not $\ol{Z}$-active. In this case, by
Proposition~\ref{prop_inactive} we have that $\ol{Z} u'' \equiv \ol{Z} v''$
implies $u'' \equiv v''$ which in turn implies $Z u'' \equiv Z v''$, so
it suffices to test the latter.

\item[\textbf{Line \ref{li_rec_case2a}.}] Here we know that
$u = XYu'$, $v = XYv'$, that $Z$ is not a prefix of $u'$ or $v'$
and that $p$ is a prefix of $X$. It follows by Lemma~\ref{lemma_eqandprefix}
that $u \equiv v$ and $p$ is a possible prefix of $u$ if and only if $u' \equiv v'$.

\item[\textbf{Line \ref{li_rec_case2b}.}] This time we know that
$u = XYu'$, $v = XYv'$ and that $p$ is a prefix of $\ol{X}$ but not
of $X$. It follows by Lemma~\ref{lemma_eqandprefix} that $u \equiv v$
and $p$ is a possible prefix of $u$ if and only if $u' \equiv v'$ and
$Z$ is a possible prefix of $u'$.

\item[\textbf{Line \ref{li_rec_case3a}.}] Here we have
$u = XYZ u''$ and $v = \ol{XYZ} v''$, and $p$ is a prefix of $X$ or $\ol{X}$.
It follows by Lemma~\ref{lemma_eqandprefix} that $u \equiv v$ and $p$ is a
possible prefix of $u$ if and only if either $Z u'' \equiv Z v''$ or
$\ol{Z} u'' \equiv \ol{Z} v''$. We also know
that $u''$ is $\ol{Z}$-active, so by Corollary~\ref{cor_actsame}, this is
true if and only if $\ol{Z} u'' \equiv \ol{Z} v''$.

\item[\textbf{Line \ref{li_rec_case3b}.}] This is the same as the previous
case, except that $u''$ is not $\ol{Z}$-active. In this case, by
Proposition~\ref{prop_inactive} we have that $\ol{Z} u'' \equiv \ol{Z} v''$
implies $u'' \equiv v''$ which in turn implies $Z u'' \equiv Z v''$, so
it suffices to test the latter.

\item[\textbf{Line \ref{li_rec_case4}.}] If we get here, we know that
$u = XY u'$, that $v = \ol{XYZ} v''$, that $Z$ is not a prefix of $u'$
and that $p$ is a prefix of $X$ or $\ol{X}$; it follows that $u \equiv v$
and $p$ is a possible prefix of $u$ if and only if condition (4') of
Lemma~\ref{lemma_eqandprefix} holds, that is, if and only if
$u' \equiv Zv''$. By the inductive hypothesis, the recursive call will
correctly estbalish if this is the case.

\item[\textbf{Line \ref{li_rec_case5}.}] The argument here is symmetric
to that for termination at line~\ref{li_rec_case4}.

\item[\textbf{Line \ref{li_rec_case6}.}] Having got here, we know that
$p$ is a prefix of $X$ or $\ol{X}$, that $u = XYu'$ and $v = \ol{XY} v'$ where $Z$ is not a prefix of $u'$
and $\ol{Z}$ is not a prefix of $v'$. We know also that $z$ is the
maximal common suffix of $Z$ and $\ol{Z}$ and $z_1$ and $z_2$ are such
that $Z = z_1 z$ and $\ol{Z} = z_2 z$. Finally, we know that
$u' = z_1 u''$ and $v' = z_2 v''$.
 It follows by Lemma~\ref{lemma_eqandprefix}
that $u \equiv v$ and $p$ is a possible prefix of $z$ if and only if
$u'' \equiv v''$ and $z$ is a possible prefix of $u''$. By the 
inductive hypothesis, the recursive call correctly establishes whether
this holds.
\end{itemize}
\end{proof}

We have now shown that our algorithm produces the correct output whenever
it terminates, but we have not yet shown that it always terminates. In
fact, the following theorem shows that it does so after only a linear
number of recursive calls.

\begin{lemma}\label{lemma_recbound}
Let $k$ be the length of the longest maximal piece suffix of a
relation word. The number of recursive calls during execution of a
call to $\proc{WP-PREFIX}(u,v,p)$ is bounded above by
$(k+2) |u| + 1$.
\end{lemma}
\begin{proof}
For clarity in our analysis, we let $u_i$, $v_i$ and $p_i$ denote the
parameters to the $i$th recursive call in the execution (with in
particular $u_0 = u$, $v_0 = v$ and $p_0 = p$). 
Each call to the function involves executing exactly one of the sections
\ref{li_start_a}--\ref{li_end_a}, \ref{li_start_b}--\ref{li_end_b}
and \ref{li_start_c}--\ref{li_end_c}; we call these calls of type A,
B and C respectively. We shall show that the number of calls of each
of these types is bounded above by a linear function of $|u|$ so that,
the total number of recursive calls is also bounded above by a linear
function of $|u|$.

First, notice that a call of type A cannot make a recursive call, so that
is only at most one type A call in the execution.

Now for a word $x$ we let $r(x) = 0$ if $x$ does not have a clean overlap
prefix, and $r(x)$ to be the length of the part of $x$ which
follows the shortest clean overlap prefix, that is, $|x'|$ where $x = aXYx'$
with $aXY$ the shortest clean overlap prefix, otherwise.

It is readily verified that if the $i$th recursive call is of type B and
itself makes a recursive call then we have $r(u_{i+1}) = r(u_i)$, while if
the $i$th recursive call is of type C and itself makes a recursive call then
we have $r(u_{i+1}) < r(u)$.
Since $r(u_i)$ can never be negative, it follows that the total number of
recursive calls of type C is linearly bounded above by $r(u_0) + 1$, which
clearly is no more than $|u_0|$.

Now note that if the $i$th recursive call is of type
B and itself makes a recursive call then we have $|u_{i+1})| = |u_i| -1$,
while if
the $i$th recursive call is of type C and itself makes a recursive call then
we have $r(u_{i+1}) \leq |u_i| + k$.

We have seen that the entire execution cannot feature more than $|u_0|$
calls of type C or more than one call of type A. Hence, if the execution
involves $i$ recursive calls, it must include at most $|u_0|$ calls of
type C, and at least $i - |u_0| - 1$ calls of type B.
It follows that, if execution involves $i$ recursive calls, we must have
$$|u_i| \ \leq \ |u_0| + |u_0| k - (i-|u_0| -1) \ = \ (k+2) |u| - i + 1$$
Since the length of $u_i$ cannot be negative, it follows that execution
must terminate after at most $(k+2) |u| + 1$ calls.
\end{proof}

It remains to justify our claim that this algorithm can be
implemented in linear time. Since the concept of linear time is highly
dependant upon model of computation, it is necessary to be precise
upon the model under consideration. We consider a Turing
machine with two two-way-infinite read-write storage tapes, using a tape
alphabet including the generators for our monoid and a separator
symbol $\#$. (Recall that a two-way-infinite tape can be simulated using
a one-way-infinite tape in linear time \cite[Section~7.5]{Hopcroft69},
so the assumption of a two-way-infinite tape is essentially immaterial).
If we assume that the input words $u$, $v$ and $p$ are initially
encoded on one of the tapes in the form $\#u\#v\#p\#$, then it is 
easily seen that, with a linear amount of preprocessing,
we can store the piece $p$ in the finite state control, and arrange for $\# u \#$
and $\# v \#$ to be the content of the first and second tape respectively.

It is straightforward to verify that, given a word $u$, one can check
whether $u$ has a clean overlap prefix of the form $XY$, and if so find
$X$, $Y$ and the corresponding $Z$, by analysing a prefix of $u$ of bounded
length. Similarly, for a given maximal piece suffix $Z$, we can check whether
$u$ is $Z$-active by analysing a prefix of $u$ of bounded length.
It follows that each recursive step of our algorithm involves analysing
prefixes of $u$ and $v$ of bounded length, before possibly making a
recursive call, with $u$ and $v$ modified only by changing prefixes of bounded length. 
Clearly any analysis of a bounded length prefix can be performed in constant
time; moreover, if a recursive call is required then the tape contents can
be modified to contain the parameters for that call, again in constant time. 
It follows that the algorithm can be implemented with execution time bounded
above by a linear function of the number of recursive calls in the execution,
which by Lemma~\ref{lemma_recbound} is bounded above by a linear function of
the length of $u$.

Moreover, by swapping $u$ and $v$ at the start of the computation if
necessary, we may assume without loss of generality that $u$ is shorter
than $v$. Thus we obtain the following.
\begin{theorem}\label{thm_lineartime}
For each every monoid presentation satisfying $C(4)$, there exists a
two-tape Turing machine which solves the corresponding word problem in
time linear in the shorter of the input words.
\end{theorem}

The reader may initially be surprised by the fact that one can test
equivalence of two words in time bounded by a function of the \textit{shorter}
word -- indeed, this bound potentially
does not even afford time to fully read the longer word! However, Remmers
showed that, for a fixed $C(3)$ presentation, the length of the longer of
two equivalent words is bounded by a linear function of the length of the
shorter \cite[Theorem~5.2.14]{Higgins92}. Thus, if the difference in lengths of two words is too great, one
may conclude without further analysis that the words are not equivalent. 
In fact Remmers' result is the only possible
explanation for this phenomemon, so the fact that this property holds for
$C(4)$ presentations can also be deduced from
Theorem~\ref{thm_lineartime}.

\section{Uniform Decision Problems}\label{sec_uniform}

In Section~\ref{sec_algorithm} we developed a linear time algorithm to solve
the word problem for a fixed small overlap presentation. Since our method
of describing the algorithm was entirely constructive, one might reasonably
expect that it also gives rise to a solution for the uniform word
problem for $C(4)$ presentations, that is the algorithmic problem of,
given a $C(4)$ presentation and two words, deciding whether the words represent the same
element of the monoid presented. In this section, we shall see that this is indeed the
case, and show that the resulting algorithm remains fast.

To avoid unnecessary technicalities, we describe and analyse the algorithms
using the RAM model of computation; in particular this allows us to
assume that elementary operations involving generators from the presentation
(such as comparing two generators) are single steps performable in constant
time. The exact time complexity of a Turing machine implementation would
depend upon the number of tapes and the precise encoding of the input, but
would certainly remain polynomial of low degree in the input size.

We begin with some simple results describing the complexity of some
elementary computations with a finite monoid presentation.
If $\langle \scrA \mid \scrR \rangle$ is a finite presentation we denote by $|\scrA|$
the cardinality of the alphabet $\scrA$, and by $|\scrR|$ the sum length of the
relation words in $\scrR$. Where the meaning is clear, we shall abuse notation
by using $\scrR$ also to denote the set of relation words in the presentation.
\begin{proposition}\label{prop_findprefix}
There is a RAM algorithm which, given a presentation $\langle \scrA \mid \scrR \rangle$
and a word $w$, computes the maximum piece prefix (and/or maximum piece suffix) of $w$
in time $O(|w| |\scrR|)$. In particular, there is a RAM algorithm to decide,
given the same inmput, decides whether the word $w$ is a piece in time
$O(|w| |\scrR|)$.
\end{proposition}
\begin{proof}
For each relation word $R \in \scrR$ and position $1 < i < |R|$ in that word
we can compute
in time $O(|w|)$ the length $n$ of the longest common prefix of $w$ and
$R_i \dots R_{|R|}$ (where $R_j$ represents the $j$th letter of $R$). Our
machine does this for each relation word and each position in that relation
word in turn, recording as it goes along (i) the maximum value
of $n$ attained so far, and (ii) the maximum value of $n$ which has been
attained or exceeded at least twice. The latter, upon completion, is clearly
the length of the longest piece prefix of $w$, and the total time taken for
execution is
$$O\left( \sum_{R \in \scrR} \ \sum_{i = 1}^{|R|} |w| \right) \ = \ O \left( |w| |\scrR| \right)$$
as claimed. An obvious dual algorithm can be used to find the longest piece
suffix of $w$.
\end{proof}

\begin{corollary}
There is a RAM algorithm which, given as input a presentation $\langle \scrA \mid \scrR \rangle$,
decides in time $O(|\scrR|^2)$ whether the presentation satisfies the condition
$C(4)$.
\end{corollary}
\begin{proof}
Our machine begins by computing the maximum piece prefix $X_R$ and maximum
piece suffix $Z_R$ for each relation word $R \in \scrR$; by
Proposition~\ref{prop_findprefix} this can be done in time
$$O \left( \sum_{R \in \scrR}  |R| |\scrR| \right) \ = \ O(|\scrR|^2).$$

It then tests, in time $O(|\scrR|)$, whether for any of the relation
words $R$ we have $|X_R| + |Z_R| \geq |R|$. If so then some relation word
is a product of two pieces, so the presentation does not even satisfy the
weaker condition $C(3)$ and we are done.

Otherwise, the machine computes, again in time $O(|\scrR|)$, the 
middle word $Y_R$
of each relation word.  By our remarks in
Section~\ref{sec_prelim}, the presentation satisfies $C(4)$ if and only
if none of the words $Y_R$ is a piece. Using Proposition~\ref{prop_findprefix}
again, this condition can be tested in time
$$O \left( \sum_{R \in \scrR} |Y_R| |\scrR| \right) \ = \ O \left( |\scrR|^2 \right).$$

Thus, we have described a RAM algorithm to test a presentation
$\langle \scrA \mid \scrR \rangle$ for
the $C(4)$ condition in time $O(|\scrR|^2)$.
\end{proof}

\begin{theorem}\label{thm_ramuniform}
There is a RAM algorithm which, given as input a $C(4)$ presentation
$\langle A \mid R \rangle$ and two words $u, v \in A^*$, decides whether
$u$ and $v$ represent the same element of the semigroup presented in
time
$$O \left( |\scrR|^2 \min(|u|,|v|) \right).$$
\end{theorem}
\begin{proof}
Suppose we are given a $C(4)$ presentation $\langle A \mid R \rangle$ and
two words $u, v \in \scrA^*$.
Just as in the proof of Proposition~\ref{prop_findprefix}, the machine
begins by finding for every relation $R$ the maximum piece prefix $X_R$,
the maximum piece suffix $Z_R$ and the middle word $Y_R$, in time
$O(|\scrR|^2)$.

It now has the information required to apply the algorithm \proc{WP-PREFIX}
given above. A simple line-by-line analysis shows that each line, and hence
each recursive call, can be executed in time $O(|\scrR|)$. By
Lemma~\ref{lemma_recbound}, the number of recursive calls is bounded above by
$(k+2)|u| + 1$ where $k$, being the length of the longest maximum piece
suffix of a relation word, is less than $|\scrR|$. Thus, this part of the
algorithm terminates in time $O(|\scrR|^2 |u|)$.

As above we may assume, by exchanging $u$ and $v$ at the start of the
computation if necessary, that
$|u| < |v|$ so that $\min(|u|, |v|) = |u|$. It follows that the uniform
word problem can be solved in time
$O \left( |\scrR|^2 \min(|u|,|v|) \right)$ as claimed.
\end{proof}

\section*{Acknowledgements}

This research was supported by an RCUK Academic Fellowship. The author is
grateful to V.~N.~Remeslennikov, whose questions prompted this line of
research and who shared many helpful ideas. He would also like to thank
A.~V.~Borovik for some helpful conversations, J.~B.~Fountain and
V.~A.~R.~Gould for facilitating access to some of the relevant literature,
and Kirsty for all her support and encouragement.

\bibliographystyle{plain}

\def\cprime{$'$} \def\cprime{$'$}

\end{document}